\documentclass[a4paper,12pt,twoside]{amsart}
\usepackage{amsmath,amsthm}
\usepackage{amsfonts}
\usepackage{amssymb}
\usepackage[english]{babel}

\newtheorem{theo}{Theorem}[section]
\newtheorem{lem}[theo]{Lemma}
\newtheorem{prop}[theo]{Proposition}
\newtheorem{cor}[theo]{Corollary}

\newcommand{\e}{{\text{e}}}

\begin{document}

\title[Doeblin's condition, $\rho$-mixing and spectra of convolutions]
{Doeblin's condition, $\rho$-mixing and spectra \\ of convolution operators on the circle}

\author{Guy Cohen}
\address{School of electrical engineering, Ben-Gurion University, Beer-Sheva, Israel}
\email{guycohen@bgu.ac.il}

\author{Michael Lin}
\address{Department of Mathematics, Ben-Gurion University, Beer-Sheva, Israel}
\email{lin@math.bgu.ac.il}


\dedicatory{\large Dedicated to the memory of Do\v gan \c C\"omez}

\begin{abstract}
We study the asymptotic behavior of Markov operators $P_\mu$ defined by convolution
with a probability measure $\mu$ on the unit circle $\mathbb T$. We prove that when
$\mu$ is adapted, $P_\mu$ satisfies Doeblin's condition if and only if some power $\mu^k$
is non-singular. We give an example of a symmetric probability measure $\mu$ on $\mathbb T$,
such that the reversible stationary chain induced by  $P_\mu$  is $\rho$-mixing, but $P_\mu$
 does not satisfy Doeblin's condition. We look at  the spectra of $P_\mu$ in the different $L_p$ 
spaces when $P_\mu$ is, or is not, $\rho$-mixing.
\end{abstract}

\maketitle

\section{Introduction}

Let $P(x,A)$ be a transition probability on a general measurable (state) space $(S,\Sigma)$,
with invariant probability $m$, i.e. $m(A) =\int_S P(x,A) dm(x)$ for $A \in \Sigma$.
The Markov operator $Pf(x):=\int_S f(y)P(x,dy)$, defined for bounded measurable $f$,
extends to all $f \in L_1(m)$, and for $1\le p\le\infty$, $P$ is also a contraction of $L_p(m)$.
We assume the chain to be ergodic, i.e. $Pf=f \in L_\infty(m)$ implies $f$ is constant a.e.
We denote by $Ef =\int_S f\,dm$ the expectation in $(S,\Sigma,m)$.
The $n$-step trasition is given by $P^{(n)}(x,A)=P^n1_A(x)$.
\smallskip

Rosenblatt \cite[p. 207]{Ro} proved that a stationary ergodic Markov chain with transition
operator $P$ is $\rho$-mixing\footnote{See definition, as "asymptotically uncorrelated
chain", in \cite[p. 206-207]{Ro}} if and only if
\begin{equation} \label{rho}
	\lim_{n \to \infty} \|P^n-E\|_2 = 0.
\end{equation}
It follows that if $P$ satisfies \eqref{rho}, so does $P^*$, and the time-reversed
chain is also $\rho$-mixing.
Rosenblatt \cite{Ro} used \eqref{rho} to prove the CLT for every centered square-integrable
additive real functional on the chain.

Rosenblatt \cite[p. 211]{Ro} proved that if $\|P^n-E\|_r \to 0$ for some $1\le r\le \infty$,
then $\|P^n-E\|_p \to 0$ for every $1<p< \infty$, and showed that $\|P^n-E\|_\infty \to 0$
is equivalent to Doeblin's condition\footnote{Which is: "There exists $n\ge 1$ and 
$\varepsilon \in(0,1)$ such that $m(A) <1-\varepsilon$ implies 
$\sup_x P^{(n)}(x,A) \le \varepsilon$."} for $P$. 
By duality, $\|P^n-E\|_1$ is equivalent to Doeblin's condition for $P^*$. 
The example in \cite[pp. 213-214]{Ro} (on a countable space) shows $\|P^n-E\|_\infty\to 0$ 
while $ \lim_n\|P^n-E\|_1= \lim_n\|P^{*n}-E\|_\infty  >0$. A different example (on $[0,1]$)
was later given in \cite[pp. 316-318]{BR}.
An example of $P$ symmetric in $L_2(m)$ with $\|P^n-E\|_2 \to 0 $, which does not satisfy
Doeblin's condition, i.e. $\lim_n\|P^n-E\|_\infty >0$, is indicated in \cite[p. 214]{Ro}.

Rosenblatt \cite[Lemma 1, p. 200]{Ro} proved that a stationary ergodic Markov chain is
$\alpha$-mixing\footnote{See definition, as "strongly mixing chain", in \cite[p. 194]{Ro}.}
if and only if 
\begin{equation} \label{alpha}
	\alpha_n:=\underset{\overset{\|f\|_\infty\le 1}{Ef=0}}\sup \|P^n f\|_1 \to 0.
\end{equation}
In general, $\rho$-mixing implies $\alpha$-mixing \cite[Lemma 3, p.200]{Ro}, but the converse 
may fail.
\smallskip
	
For $P$ a convolution operator on the unit circle, we study in this note the convergence 
$\|P^n-E\|_p \to 0$ and its possible relation with the spectra of $P$ in the different 
$L_p$ spaces. We present a symmetric $P$ with $\|P^n-E\|_2 \to 0$ and 
$\lim_n \|P^n-E\|_\infty>0$.

\medskip

\section{Convolution operators on the unit circle}

Let $\mathbb T$ be the circle group, with Borel field $\mathcal B$ and normalized
Lebesgue measure $m$. A probability measure $\mu$ on $(\mathbb T,\mathcal B)$ defines a
transition probability $P(x,A):=\mu(Ax^{-1})$, with Markov operator
$P_\mu f(x)=\int f(xy)\mu(dy)$ and invariant probability $m$; thus
$P_\mu f=\check\mu*f$, where  $\check\mu(A)=\mu(A^{-1})$ for $A \in\mathcal B$.
Since $\mathbb T$ is an Abelian group, $P_\mu$ is a normal operator on $L_2(\mathbb T,m)$.
When $\mu$ is symmetric (i.e. $\check\mu=\mu$), $P_\mu$ is symmetric in $L_2(\mathbb T,m)$,
which means that the stationary Markov chain it generates (a random walk on $\mathbb T$)
is reversible.  When $\mu$ is given, we denote $P_\mu$ by $P$. Note that $P_\mu$ is also
a contraction of $C(\mathbb T)$, the space of continuous functions on $\mathbb T$.
\smallskip

Let $\mu$ be {\it adapted} -- the closed subgroup generated by its support is $\mathbb T$;
then $P_\mu$ is ergodic.  Rosenblatt \cite[p. 202]{Ro} proved that
\begin{equation} \label{rho-fourier}
\|P_\mu^n-E\|_2 \to 0 \quad \text{if and only if} \quad 
\sup_{0\ne n\in \mathbb Z} |\hat\mu(n)|<1;
\end{equation}
see also  \cite[Corollary 2.2(i)]{DL4}.
It was shown in \cite[Theorem 4.4]{DL4} that if $\hat\mu(n) \to 0$ as $|n| \to \infty$
(i.e. $\mu$ is {\it Rajchman}), then $\|P_\mu^n-E\|_2 \to 0$.

 \begin{prop}
	 Let $\mu$ be an adapted probability on $\mathbb T$. Then the Markov chain generated
	 by $P_\mu$ is $\rho$-mixing if  (and only if) it is $\alpha$-mixing.
 \end{prop}
\begin{proof} Assume $\alpha$-mixing. Fix $n$ such that $\alpha_n$, defined in \eqref{alpha},
is less than 1. Put $e_k(z)=z^k$, $z \in\mathbb T$. Then $P_\mu e_k = \hat\mu(-k)e_k$.
For $0 \ne k\in\mathbb Z$ we then have
$$
1>\alpha_n \ge \|P_\mu^n e_k\|_1= |\hat\mu(-k)|^n \|e_k\|_1= |\hat\mu(-k)|^n.
$$
Hence $\sup_{0\ne k\in\mathbb Z}|\hat\mu(k)| \le \alpha_n^{1/n} <1$, so by 
\eqref{rho-fourier} $\|P_\mu^n-E\|_2 \to 0$, which is equivalent to $\rho$-mixing.

The converse is proved in \cite[p. 200]{Ro}.
\end{proof}

By \cite[Proposition 2.1]{ConL}, the strong convergence $\|P_\mu^nf-Ef\|_2 \to 0$ for every
$f \in L_2(\mathbb T,m)$ is equivalent to $|\hat\mu(n)|<1$ for every $0\ne n \in\mathbb Z$
({\it strict aperiodicity}).  By \cite[Proposition 4.5]{DL4}, if $\mu$ is a discrete strictly
aperiodic probability, then $\sup_{n\ne 0}|\hat\mu(n)|=1$, so $\lim_n \|P_\mu^n-E\|_2 >0$.
When $\mu=\delta_z$ with $z \in \mathbb T$ not a root of unity, then $\mu$ is adapted
with $|\hat\mu(n)|=1$ for every $n$; hence $\lim_n\|P_\mu^n - E\|_p > 0$ for every
$1 \le p \le \infty$ (which is easily checked directly).

When $\mu$ is non-singular, write $\mu=\nu_a+\nu_s$, with $0\ne \nu_a\ll m$.
Then $\mu$ is adapted, and, by the Riemann-Lebesgue lemma,
 $\limsup_{|n| \to \infty} |\hat \mu(n)| \le \nu_s(\mathbb  T)<1$. 
Since the argument in the proof of \cite[Theorem 4.5]{DL4} shows that $|\hat\mu(n)|<1$ 
for every $n\ne 0$, we obtain $\sup_{n \ne 0} |\hat\mu(n)|<1$, and 
$\lim_n\|P_\mu^n-E\|_2= 0$. This is strengthened in Theorem \ref{harris} below.

\begin{lem} \label{strict}
A non-discrete probability on $\mathbb T$ is adapted and strictly aperiodic.
\end{lem}
\begin{proof} Let $\mu$ be  non-discrete. Then its support $S_\mu$ is uncountable,
hence contains some $z$ not a root of unity. Since $(z^n)_{n>0}$ is dense in
$\mathbb T$, $\mu$ is adapted.

$\mu$ is strictly aperiodic if and only if the closed group generated by 
$S_\mu\cdot S_\mu^{-1}$ is $\mathbb T$ (e.g. \cite{ConL}). Since $S_\mu$ is uncountable, 
so is $S_\mu\cdot S_\mu^{-1}$; hence, as above, the closed subgroup the latter
generates is $\mathbb T$.
\end{proof}

{\bf Remark.} If $z \in \mathbb T$ is not a root of unity, then $\delta_z$ is adapted, but not 
strictly aperiodic.

\begin{theo} \label{harris}
Let  $\mu$ be a probability on $\mathbb T$ such that $\mu^k$ is non-singular for some 
$k\ge 1$.  Then $\|P_\mu^n-E\|_1 \to 0$ and $\|P_\mu^n-E\|_\infty \to 0$. 
Moreover, $\|P_\mu^n-E\|_p \to 0$ for every $1 \le p \le \infty$.
\end{theo}
\begin{proof}
Since $\mu^k$ is non-singular, it is adapted. Hence $P_\mu^k$ is ergodic; hence so is $P_\mu$.

Since $\mu^k$ is non-singular, so is $\mu^{*k}$; hence  there exists a positive measure
$0\ne \nu \ll m$ such that $\nu \le \mu^{*k}$. Put $q=d\nu/dm$ and define $Qf=q*f$ for
$f \in L_1(\mathbb T)$. By \cite[Exercise 55, p. 518]{DS} (see \cite{CG} for a proof in
general compact groups), $Q \ne 0$ is a compact operator on $L_1(\mathbb T)$, and
$\|P_\mu^k-Q\|_1=\|\mu^{*k}-\nu\|_1 <1$, i.e. $P_\mu$ is quasi-compact in $L_1(\mathbb T)$.
By \cite[Corollary VIII.8.4]{DS}, $\big\|\frac1n\sum_{j=1}^nP_\mu^j-E\big\|_1 \to 0$.

By \cite[Theorem VIII.8.3]{DS}, $\sigma(P_\mu,L_1)\cap\mathbb T$ is finite and consists
only of eigenvalues (simple poles).
By \cite[Proposition 2.1]{CoL} there exists $d \ge 1$ such that $\lambda^d=1$ whenever
$\lambda \in \sigma(P_\mu,L_1)\cap\mathbb T$. Since $\mu^{kd}$ is also non-singular,
$P_\mu^{kd}$ is ergodic. Let $\lambda \in \sigma(P_\mu,L_1)\cap\mathbb T$, with eigenfunction
$0\ne g \in L_1(\mathbb T)$. Then $P_{\mu}^{kd}g =\lambda^{kd}g =g$, so $g$ is constant,
and $\lambda=1$. Hence $\|P_\mu^n-E\|_1 \to 0$ by \cite[Proposition 3.3]{DL4}.

Applying the above to $\check\mu$, we obtain
$\|P_\mu^n-E\|_\infty =\|P_\mu^{*n}-E\|_1= \|P_{\check\mu}^n-E\|_1 \to 0$.

For $1<p< \infty$, the convergence $\|P_\mu^n-E\|_p \to 0$ now follows from \cite[p. 211]{Ro}.
\end{proof}

{\bf Remarks.} 1. For $\mu$ non-singular, $\|P_\mu^n-E\|_1 \to 0$ follows from 
\cite[Theorem 3]{Bh}.

2. Brunel and Revuz \cite[Proposition VIII.3]{BR} proved, for any compact metrizable group
$G$, that if $\mu$ is an adapted probability with  a non-singular power, then $P_\mu$ is 
quasi-compact on the space $B(G)$ of all bounded Borel measurable functions (with the sup norm). 
This is (formally) stronger than quasi-compactness on $L_\infty(G)$, obtained (by duality) 
in Theorem \ref{harris} for $G =\mathbb T$.

3. Borda \cite[Theorem 1]{Bo} characterized the property that $\mu$ is adapted with some 
power non-singular (in any compact metric group $G$) by a Law of Iterated Logarithm for 
every bounded Borel function on $G$. See \cite{Bo} for more properties and earlier references.
\medskip

We show that in general, convergence of $P_\mu^n$ in $L_1$ operator norm need not hold.

\begin{prop} \label{singular}
Let $\mu$ be an adapted probability on $\mathbb T$.
If all convolution powers of $\mu$ are singular, then $\lim_n \|P_\mu^n-E\|_1 >0$
and $\lim_n\|P_\mu^n-E\|_\infty > 0$.
\end{prop}
\begin{proof}
Assume $\lim_n \|P_\mu^n-E\|_1 =0$. Then $\|P_\mu^{*n}-E\|_\infty \to 0$.
By \cite[Theorem 4.1(g)]{Ho}, $P_\mu^*$ is Harris, so some power bounds a non-zero
integral operator (e.g., \cite[Section 6.2]{Re}).  But $P_\mu^*f=\mu*f$, so some power
$\mu^n$ must bound a non-zero absolutely continous measure, contradicting the singularity
of all powers of $\mu$.

Also all convolution powers of $\check\mu$ are singular, so applying the above to
$\check \mu$ yields $\lim_n\|P_\mu^n-E\|_\infty =\lim_n\|P_\mu^{*n}-E\|_1 =
\lim_n\|P_{\check\mu}^n-E\|_1	> 0$.
\end{proof}

\begin{cor} \label{doeblin}
Let $\mu$ be an adapted probability measure on $\mathbb T$.

(i) $P_\mu$ satisfies Doeblin's condition if and only if some power $\mu^k$ is non-singular.

(ii) $P_\mu$ satisfies Doeblin's condition if and only if so does $P_\mu^*$, i.e.
$\|P_\mu^n-E\|_\infty \to 0$ if and only if $\|P_\mu^n-E\|_1 \to 0$.
\end{cor}

{\bf Remark.} For general compact groups, Galindo et al. \cite[Theorem 4.7]{GJR}
proved that when $\mu$ is adapted, $\|\frac1n\sum_{k=1}^n P_\mu^k-E\|_1 \to 0$ 
if and only if some power $\mu^k$ is non-singular. They noted also that then $P_\mu$
is quasi-compact on $L_1$. The "if" part had been proved differently in \cite{BR}.

\begin{prop} \label{no-rho}
There exists a continuous probability measure $\mu$ on $\mathbb T$, with all its
convolution powers singular, such that ${\lim_n \|P_\mu^n-E\|_p > 0}$ for every
$1\le p\le \infty$.
\end{prop}
\begin{proof} Let $(n_k)_{k\ge 1}$ be an increasing sequence of natural numbers such that
$n_k$ divides $n_{k+1}$ for $k \ge 1$. Combining \cite[Propositions 3.9 and 3.1]{EG},
we obtain a continuous probability measure $\mu$ on $\mathbb T$ such that
$\hat\mu(n_k) \to 1$. Hence $\sup_{0\ne n \in\mathbb Z} |\hat\mu(n)|=1$, and thus
$\lim_n\|P_\mu^n-E\|_2 > 0$ by \eqref{rho-fourier}.
By \cite{Ro} $\lim_n\|P_\mu^n-E\|_p > 0$ for every $1 \le p \le \infty$.

If $\mu$ is non-singular, then $\limsup |\hat\mu(n)| \le \nu_s(\mathbb T)<1$, where
$\nu_s$ is the singular component, a contradiction.

Since $|\hat\mu^j(n_k)| = |\hat\mu(n_k)|^j \to 1$, also $\mu^j$ is singular.
\end{proof}

{\bf Remarks.} 1. The probability measure $\mu_0:= \sum_{j=1}^\infty 2^{-j}\mu^j$ is
continuous with support $\mathbb T$ (closure of the union of the supports of the powers
$\mu^j$), and $\hat\mu_0(n_k) \to 1$.

2. In the proof of Proposition \ref{no-rho} we can also use  sequences $(n_k)$ with
$n_{k+1}/n_k \to \infty$ \cite[Example 3.4]{EG}.  Other sequences are given by
\cite[Theorem 1]{FT}.

3. For two specific constructions, see \cite{EI}.

4. By \cite[Theorem 4.6]{DL4}, every adapted discrete $\mu$ satisfies 
${\lim_n \|P_\mu^n-E\|_2 > 0}$; hence ${\lim_n \|P_\mu^n-E\|_p > 0}$ for $p\ge 1$.
A simple example, with $\sup_{n\ne 0}|\hat\mu(n)|=1$ easily computed directly, is 
$\mu=\frac12(\delta_z+\delta_{z^{-1}})$, with $z \in \mathbb T$ not a root of unity.
\medskip

We denote by $\sigma(P,L_p)$ the spectrum of $P_\mu$ as an operator on $L_p(\mathbb T,m)$,
$1 \le p \le \infty$.  Since $\{e_n(z):=z^n: n\in\mathbb Z\}$ is an orthonormal basis
of eigenfunctions, $\sigma(P,L_2)= \overline{\{\hat\mu(n): n\in\mathbb Z\}}$.
It follows also that $ \overline{\{\hat\mu(n); n\in\mathbb Z\}} \subset \sigma(P,L_p)$ for
$1 \le p \le \infty$. Sarnak \cite[p. 309]{Sa} deduced from his main result that if
$\overline{\{\hat\mu(n); n\in\mathbb Z\}}$ is countable, in particular if $\mu$ is Rajchman,
then $\sigma(P,L_p)=\sigma(P,L_2)$ for $1<p< \infty$. When $\mu$ is absolutely continuous
(with respect to $m$) then $\sigma(P,L_p)=\sigma(P,L_2)$ for $1 \le p \le\infty$
\cite[Theorem 6]{Ba}.

Since $P_\mu$ is normal, $\|P_\mu^n-E\|_2 \to 0$ is equivalent, by the spectral theorem,
to $1$ isolated in $\sigma(P,L_2)$ and no other unimodular points in the spectrum.
By \cite{Ro},  $\|P_\mu^n-E\|_p \to 0$ for $1<p<\infty$, so $1$ is a pole of the
resolvent of $P_\mu$ in $L_p$, hence isolated in $\sigma(P_\mu,L_p)$, with no other
unimodular points in that spectrum.

Galindo and Jord\'a \cite[Theorem 6.1]{GJ} proved, for $\mu$ adapted on a general
 compact Abelian group and $1 \le p \le \infty$, that
$\|\frac1n\sum_{k=1}^n P_\mu^k -E\|_p \to 0$ if (and only if) $1$ is isolated in 
$\sigma(P_\mu,L_p)$ (see also \cite[Lemma 2.2]{Za}). It follows (see \cite{DL4}) that 
$\|P_\mu^n-E\|_p \to 0$ if (and only if) $1$ is isolated in $\sigma(P_\mu,L_p)$, with no 
other unimodular points in that spectrum. 
\medskip

For completeness we combine the above mentioned consequences of \cite{DL4} and \cite{Sa}.
\begin{theo} \label{rajchman}
Let $\mu$ be a Rajchman probability measure on $\mathbb T$. Then, for $1<p< \infty$,
$\ \|P_\mu^n-E\|_p \to 0$, and
$\sigma(P_\mu,L_p)=\sigma(P_\mu,L_2) = \overline{\{\hat\mu(n); n\in\mathbb Z\}}$.
\end{theo}

{\bf Remarks.} 1. If $\mu$ is Rajchman and $\sigma(P_\mu,L_1)=\sigma(P_\mu,L_2)$, then, by
\cite[Theorem 3.6]{Za}, some power $\mu^k$ is not singular, and by Theorem \ref{harris} 
$\|P_\mu^n-E\|_\infty \to 0$, i.e. $P_\mu$ satisfies Doeblin's condition.

2. Recall that a Rajchman probability is continuous, by Wiener's Lemma \cite[Theorem III.9.6]{Zy}.

3. We note that the convergence to zero of $\hat\nu(n)$ when $\nu$ is a Rajchman probability 
can be arbitrarily slow. This follows from \cite[Theorem 5]{BM}, by taking $d\nu=f\,d\mu$ there.

\begin{prop} \label{zafran}
There exists a Rajchman probability measure $\mu$ on $\mathbb T$ such that
	${\|P^n-E\|_\infty \to 0}$, but $\sigma(P_\mu,L_1) \ne \sigma(P_\mu,L_2)$.
\end{prop}
\begin{proof}
Let $\nu$ be a Rajchman probability measure with all its powers $\nu^k$ singular
(e.g. \cite{CoL}). Fix $0<a<1$ and define $\mu:=a\nu+(1-a)m$. Then $\mu$ is Rajchman.
Since $\mu$ is not singular, $\|P^n-E\|_\infty \to 0$, by Theorem \ref{harris}.

Assume $\sigma(P_\mu,L_1) = \sigma(P_\mu,L_2)$. Since $\mu$ is Rajchman and $\nu\ll \mu$, 
by \cite[Theorem 3.12]{Za} also $\sigma(P_\nu,L_1) = \sigma(P_\nu,L_2)$. But then, by 
\cite[Theorem 3.6]{Za}, $\nu$ has a non-singular convolution power, contradiction. Hence
$\sigma(P_\mu,L_1) \ne \sigma(P_\mu,L_2)$.
\end{proof}

\begin{prop} \label{hyper}
Let $\mu$ be an adapted probability on $\mathbb T$ such that
$P_\mu L_r(\mathbb T) \subset L_s(\mathbb T)$ for some $1 \le r <s $
($\mu$ is $L_r$-improving\footnote{ This  "moment improving" property of Markov
operators is called {\it hyperboundedness}; see \cite{CoL}.}). Then, for $1<p<\infty$,
$\ \|P_\mu^n-E\|_p \to 0$, and
$\sigma(P_\mu,L_p)=\sigma(P_\mu,L_2) = \overline{\{\hat\mu(n); n\in\mathbb Z\}}$.
\end{prop}
\begin{proof} Since $\mu$ is adapted, $P_\mu$ is ergodic.
By \cite[Corollaries 2.5 and 3.8]{CoL}, $P_\mu$ is uniformly mean ergodic in
$L_2(\mathbb T,m)$, i.e. $\|\frac1N\sum_{k=1}^N P_\mu^k-E\|_2 \to 0$.
By \cite[Theorem 4.6]{DL4} $\mu$ is not discrete. 
By Lemma \ref{strict}, $\mu$ is strictly aperiodic, hence has no unimodular
eigenvalues except 1. By \cite[Proposition 3.3]{DL4}, $\|P_\mu^n-E\|_2 \to 0$, and by
\cite{Ro} $\|P_\mu^n-E\|_p \to 0$ for every $1<p< \infty$.

The equality of the spectra is proved in \cite[Theorem 4.1]{GHR}.
\end{proof}

{\bf Remark.} There exist Rajchman probabilities which are not $L_p$-improving; see
 \cite{CoL}, Propositions 6.3 (for a probability with all powers singular) and 6.8 (for an
absolutely continuous probability). On the other hand, the exists a singular probability which
is not Rajchman and is $L_p$-improving \cite[Proposition 6.6]{CoL},\cite[Proposition 4.7]{DL4}.
\smallskip

{\bf Example.} {\it Cantor-Lebesgue measures with constant dissection rate}

\noindent
Christ \cite{Ch} proved that Cantor-Lebesgue measures with constant dissection rate
$\theta >2$ are $L_p$-improving for every $1<p<\infty$. The equality of the spectra
 in this case follows also by combining \cite[Corollary 1.4]{SS} and \cite{Sa}.
In the classical case $\theta=3$, the Cantor-Lebesgue measure has singular powers and
is not Rajchman (e.g. \cite{DL4}).
\medskip

We show that Theorem \ref{rajchman} may fail for $p=1$.

\begin{theo}\label{no-doeblin}
There exists a symmetric Rajchman probability measure $\mu$ on the unit circle $\mathbb T$,
with singular  convolution powers, such that $\sigma(P_\mu,L_1) \subset \mathbb R$, 
$\ \|P_\mu^n-E\|_p \to 0$ for every $1<p< \infty$, but $\lim_n \|P_\mu^n -E\|_1 >0$. 
 Moreover, $\sigma(P_\mu,L_p)=\sigma(P_\mu,L_2)$ for $1<p<\infty$, but
$\sigma(P_\mu,L_1)\ne \sigma(P_\mu,L_2)$.
\end{theo}
\begin{proof}
 Parreau \cite[p. 322]{Pa} constructed a symmetric Rajchman  probability measure
$\mu_0$ on $\mathbb T$ with singular powers, such that $\sigma(P_{\mu_0},L_1)$,
which equals $\sigma(\mu_0,M(\mathbb T))$ (see beginning of the proof of  Proposition
 \ref{sato}), is a subset of $\mathbb R$. By Theorem \ref{rajchman}, $\sigma(P_{\mu_0},L_p)=\sigma(P_{\mu_0},L_2)$ for $1<p<\infty$. Since $\mu_0$ 
is Rajchman, it is continuous, and therefore adapted. By \cite[Theorem 4.4]{DL4}, 
$\|P_{\mu_0}^n-E\|_2 \to 0$, and by \cite{Ro} $\|P_{\mu_0}^n-E\|_p \to 0$ for any 
$1<p< \infty$.   By Proposition \ref{singular}, $\lim_n \|P_{\mu_0}^n-E\|_1 >0$.
 Applying the result to $\check\mu_0=\mu_0$, we obtain
$\lim_{n\to \infty}\|P_{\mu_0}^n-E\|_\infty  >0$.

Ohrysko and Wojciechowski \cite[Theorem 4.3]{OW} proved that $\sigma(P_{\mu_0},L_1)$,
 is the union of an interval, $[-c,1]$ or $[-1,c]$ for some $c \in[0,1]$, with a 
finite set of reals. Hence $\sigma(P_{\mu_0},L_1)$ is uncountable. Since $\mu_0$ is 
Rajchman, $\sigma(P_{\mu_0},L_2)= \{\hat\mu_0(n): n \in \mathbb Z\}\cup\{0\}$,
 which is countable, so $\sigma(P_{\mu_0},L_1)\ne \sigma(P_{\mu_0},L_2)$.
The inequality of the spectra follows also from \cite[Theorem 3.6]{Za}; our proof shows how 
large the difference of the spectra is.

In fact, $\sigma(P_{\mu_0},L_1)$, is the union of an interval, $[-c,1]$  for some $c \in[0,1]$
with a finite set of reals; the other form given in \cite{OW}, with $[-1,c]$, $c<1$, implies 
that 1 is isolated in the spectrum, and by \cite[Theorem 6.1]{GJ} 
$\|\frac1n \sum_{k=1}^n P_{\mu_0}^k -E\|_1 \to 0$. By \cite{L} $P_\mu$ is quasi-compact
 in $L_1(\mathbb T,m)$, and strict aperiodicity yields $\|P_{\mu_0}^n -E\|_1 \to 0$, as in
 the proof of Theorem \ref{harris}, which contradicts Corollary \ref{doeblin}.
\end{proof}

{\bf Remarks.} 1. The classical Cantor-Lebesgue probability is continuous and symmetric, all
its powers are singular, and $\|P^n-E\|_2 \to 0$ \cite[Proposition 4.7]{DL4}. By Proposition
 \ref{singular} $\lim_n\|P^n-E\|_1 >0$, and $\sigma(P,L_p)=\sigma(P,L_2)$ for $1<p<\infty$
 by \cite[Theorem 3.7]{Sa}. As noted in \cite{Sa}, $\sigma(P,L_1) \ne \sigma(P,L_2)$.
 However, this measure is not Rajchman.

2. Since $\mu$ in Theorem \ref{no-doeblin} is symmetric. also $\lim_n \|P_\mu^n-E\|_\infty >0$;
thus the chain is $\rho$-mixing, but $P_\mu$ fails  Doeblin's condition.

3. Also $\mu_0^2$ satisfies the properties of Theorem \ref{no-doeblin}, and since 
$P_{\mu_0} ^2=P_{\mu_0^2}$, the spectral mapping theorem yields $\sigma(P_{\mu_0^2},L_1)=[0,1]$.
\medskip

For  completeness we recall in the following some facts tacitly used in \cite{Sa}.
\begin{prop} \label{sato}
There exists a singular continuous symmetric probability measure $\mu$ on $\mathbb T$, with 
singular powers, such that $\sigma(P_\mu,L_p) \ne \sigma(P_\mu,L_2) $ for every $1 \le p <2$.
\end{prop}
\begin{proof}For fixed $1 \le p <\infty$ let $M_p:=M_p(\mathbb T)$ be the algebra of
$L_p(\mathbb T)$ multipliers\footnote{A multiplier is a bounded $T$ on $L_p$ for which
there exists $(m(n))_{n \in \mathbb Z}$ such that $\widehat{Tf}(n)=m(n)\hat f(n)$ for
every $f \in L_p$.}.
Every convolution operator belongs to $M_p$, and $M_p$ is a commutative Banach algebra
with the operator norm. Applying \cite[Corollary 1.1.1]{La} to $L_p$ (which is a Banach
algebra with convolution as product), we obtain $\sigma(P_\mu,M_p)=\sigma(P_\mu,L_p)$.
Note that $M_1(\mathbb T)$ is isometrically isomorphic to $M(\mathbb T)$, the space 
of all finite complex Borel measures on $\mathbb T$ \cite[Theorem 0.1.1]{La}.

Sato \cite[Theorem 1]{Sato} proved the existence of a symmetric probability measure $\mu$ 
on $\mathbb T$ such that for any $1 \le p <2$,
$$
\sigma(P_\mu,L_2)=\overline{\{\hat\mu(n): n\in \mathbb Z\}} \ne \sigma(P_\mu,M_p)=
\sigma(P_\mu,L_p).
$$
By Theorem \ref{rajchman} this $\mu$ is not Rajchman (this is also seen directly in 
\cite{Sato} from the formula for $\hat\mu$).
\smallskip

We now prove the continuity of Sato's $\mu$.
	
Sato \cite[Case 2, p. 336]{Sato} defined (using his notation) the measure $\mu$ by a weak* 
limit of probability measures of the form $d\mu_n(t)=p_n(t)dm$, where 
$p_n(t)=\Pi_{j=1}^n\phi_j(r_j t) $, with $\phi_j$ non-negative trigonometric polynomials, 
satisfying $\hat\phi_j(0)=1$ and $\hat\phi_j(k)=\hat\phi_j(-k) \ge 0$ for $k\ge1$,
and $r_j\uparrow\infty$ is a rapidly increasing sequence of natural numbers.
\smallskip
	
By the assumptions on $\phi_j$, it has the representation 
$\phi_j(t)=1+2\sum_{k=1}^{m_j} \hat\phi_j(k)\cos(kt)$, for some $m_j$. By the definition 
of $\mu_n$, for each $n$,
$$
\hat\mu(r_1 k_1+\cdots+ r_n k_n)= \hat\mu_n(r_1 k_1+\cdots+ r_n k_n)=
\Pi_{j=1}^n\hat\phi_j(k_j) \quad \text{for} \  |k_j|\le m_j, \quad
1\le j\le n
$$ 
(and $\hat\mu(\ell)=0$ for other $\ell$).
We may choose the rapidly increasing $(r_j)$ (as implicitly done in Sato) in such a way 
that there will be a positive gap (in order to avoid overlap) between the maximal 
frequency\footnote{The frequencies of a Fourier series $\sum_k a_k \e^{ikt}$ are those 
$k$ for which $a_k \ne 0$.} of $p_n$, which is $r_1m_1+\cdots + r_n m_n$, and the minimal 
frequencies added to $p_{n+1}$ when multipying $p_n$ by $\phi_{n+1}$; that is, we assume 
	\begin{equation} \label{gap}
r_1m_1+\cdots + r_n m_n <{r_{n+1}-(r_1 m_1+\cdots+r_n m_n}) 
	\end{equation}
($m_j$ is defined by $\phi_j$ only). 
The classical Riesz product, as handled in \cite[ Vol I, page 208]{Zy},
	is the special case where all $m_j=1$ and $(r_j)$ is a lacunary sequence.

Put $M_n= r_1m_1+\cdots + r_n m_n$; then the partial products $p_n(t)$ can be 
represented as partial Fourier series $1+\sum_{\ell=1}^{M_n} \gamma_\ell \cos(\ell t)$, 
where $\gamma_\ell=\hat\mu(\ell)$.
By the construction we may define formally, and in a unique way, the infinite Fourier series 
$1+\sum_{\ell=1}^\infty \gamma_\ell \cos(\ell t)$, which  represents the infinite product 
 $\Pi_{j=1}^\infty\phi_j(r_jt)$.

We conclude that along $M_n$, the partial sums of the formal series are non-negative. 
We then apply Zygmund's \cite[Theorem IV.5.20, vol. I, p. 148] {Zy} to conclude that the 
limiting measure $\mu$ of $\mu_n$ is continuous. 
\medskip

To show that $\mu$ is singular we proceed as in the proof of Zygmund's 
\cite[Theorem V.7.6, Vol.I,  p. 209] {Zy}.
We  use the the inequality $1+u\le {\rm e}^u$, to conclude that 
$$
0\le p_n(t) \le \exp\big(\sum_{\ell=1}^{M_n}\gamma_\ell \cos(\ell t)\big).
$$
In choosing $(r_n)$, we assume (a strengthenning of \eqref{gap}) that for some $q >1$, 
$$
\frac{{r_{n+1}-(r_1 m_1+\cdots+r_n m_n})}{ r_1m_1+\cdots + r_n m_n}\ge q, \qquad \forall n.
$$
 We necessarily have $\sum_{\ell=1}^\infty \gamma_\ell^2=\infty$, since $\mu$ is not Rajchman.
By the infinitely many gaps property we may apply \cite[Theorem III.1.27, Vol. I, p 79]{Zy} 
and finish as in \cite[Theorem V.7.6,  p. 209] {Zy}. Hence $\mu$ is singular.
\medskip

 All powers $\mu^n$ are obviously continuous, but  not Rajchman since $\mu$ is not. 
The Fourier-Stieltjes series of $\mu^n$ is $1+\sum_{\ell=1}^\infty \gamma^n_\ell \cos(\ell t)$,
and we may repeat the process on $(\phi_j^n)$ to conclude that all powers 
$\mu^n$ are singular.
\end{proof}

{\bf Remarks.} 1. In \cite[Corollary 2.7]{KS}, Kanjin and Sato proved that $\mu$ of Proposition 
\ref{sato} satisfies also 
$\sigma(P_\mu,L_1)=\{z\in\mathbb C: |z| \le 1\} \supsetneqq \sigma(P_\mu,L_p)$, $1<p<2$.

2. In contrast to the inequality of the spectra in Proposition \ref{sato}, the peripheral 
spectra $\sigma(P_\mu,L_p)\cap\mathbb T$ are all the same for $1<p<\infty$, by the general 
result in \cite[Theorem 4.3]{CL26}.  This is not true for $p=1$; since $\mu$ is symmetric, 
$\sigma(P_\mu,L_2)\cap\mathbb T\subset\{-1,1\}$, while
$\sigma(P_\mu,L_1)\cap\mathbb T=\mathbb T$ by \cite[Corollary 2.7]{KS}.

 \begin{theo} \label{ac-power}
Let $\mu$ be a probability on $\mathbb T$ such that $\mu^k$ is absolutely continuous for
some $k\ge 1$.  Then $\mu$ is Rajchman, $\|P_\mu^n-E\|_p \to 0$ for $1\le p \le \infty$, and
 \begin{equation} \label{spectra}
\sigma(P_\mu,L_p)=\sigma(P_\mu,L_2) \quad \text{for } 1 \le p \le \infty.
 \end{equation}
\end{theo}
\begin{proof} The convergence $\|P_\mu^n-E\|_p \to 0$ for $1\le p\le \infty$ follows from
Theorem \ref{harris}.

When $k=1$,  the equality \eqref{spectra} is by \cite{Ba}, and $\mu$ is Rajchman by the
Riemann-Lebesgue lemma.

Assume $k \ge 2$. Since $(\hat\mu(n))^k= \widehat{\mu^k}(n) \to 0$, also $\mu$ is Rajchman.
Hence
\begin{equation} \label{L2-spectrum}
\sigma(P_\mu,L_2)=\overline{\{\hat\mu(n): n\in\mathbb Z\} }=
\{\hat\mu(n): n\in\mathbb Z\} \cup \{0\}.
\end{equation}
Since $\mu$ is Rajchman, for any $a>0$ there are only finitely many $n$ with $|\hat\mu(n)|=a$.

For convenience we write $P$ for $P_\mu$. Fix $2 \ne p \in [1,\infty]$. As noted earlier,
$\sigma(P,L_2)=\overline{\{\hat\mu(n): n\in\mathbb Z\} } \subset \sigma(P,L_p)$.
Let  $\lambda \in \sigma(P,L_p)$. If $\lambda=0$ then $\lambda \in \sigma(P,L_2)$,
so we assume $\lambda \ne 0$. Fix $q\ge k$.  Then also $\mu^q$ is absolutely continuous,
so by \cite{Ba} $\sigma(P^q,L_p)=\sigma(P^q,L_2)$ (for any $1 \le p \le \infty$).
By the spectral mapping theorem, $\lambda^q \in \sigma(P^q,L_p)=\sigma(P^q,L_2)$.
Since $\lambda \ne 0$, by \eqref{L2-spectrum} there exists
$\hat\mu(n_q)=\lambda_q \in \sigma(P,L_2)$ such that $\lambda_q^q=\lambda^q$, so
$|\hat\mu(n_q)|=|\lambda_q|=|\lambda|$.
	
But there are only finitely many $j_1,\dots,j_J$ such that $|\hat\mu(j_\ell)|=|\lambda|$.
Hence there is $\ell$ such that $(\hat\mu(j_\ell)/\lambda)^q=1$ for infinitely many prime $q$.
Put $z=\hat\mu(j_\ell)/\lambda$, and let $d$ be the minimal natural number with $z^d=1$,
	Assume $z \ne 1$ ($d>1$). Fix a prime $q>\max\{d,k\}$ with $z^q=1$.
Writing $q=sd+r, \ 1\le r <d$ we obtain $1=z^q=z^{sd+r}=z^r$, contradicting the minimality
of $d$. Hence $\lambda=\hat\mu(j_\ell) \in \sigma(P,L_2)$, which proves \eqref{spectra}.
\end{proof}

{\bf Remarks.} 1.  There exist singular probabilities $\mu$ such that $\mu*\mu$ is
absolutely continuous; see \cite{HZ} and \cite{eA} and references there.

2. Zafran \cite[Theorem 3.9]{Za} proved that when $\mu$ is a Riesz product defined by
${\prod_{k=1}^\infty (1+a_k\cos n_k x)}$, where $n_{k+1} \ge qn_k$ for some $q>3$ and
$0 \ne a_k \to 0$ (so $\mu$ is Rajchman), then $\sigma(P,L_1)=\sigma(P,L_2)$ if and only
if some power $\mu^n$ is absolutely continuous (i.e. $\sum_{k=1}^\infty |a_k|^n < \infty$).

3. The equality $\sigma(P_\mu,L_1)=\sigma(P_\mu,L_2)$ is equivalent to \eqref{spectra}, by
\cite[Proposition 3]{Ba}.

\begin{theo} \label{sato2}
There exists a symetric non-singular continuous probability measure $\mu$ on $\mathbb T$
 such that $\sigma(P_\mu,L_p) \ne \sigma(P_\mu,L_2)$ for every $1\le p <2$.
\end{theo}
\begin{proof} Let $\nu$ be the symmetric continuous probability constructed by Sato such
that $\sigma(P_\nu,L_p) \ne \sigma(P_\nu,L_2)$ for $1\le p <2$ (Proposition \ref{sato}).
We proved in Proposition \ref{sato} that all powers of $\nu$ are singular.

 Fix $0< a < 1$ and put $\mu:=a\nu+(1-a)m$, so $\mu$ is symmetric, continuous and 
non-singular. By Theorem \ref{harris} $\|P_\mu^n-E\|_p \to 0$ for $1\le p \le \infty$.
Hence 1 is isolated in $\sigma(P_\mu,L_p)$, and $I-P_\mu$ is invertible on
$L_p^0:=\{f\in L_p: \int f\,dm=0\}$. Put $Q_{\mu,p}:=P_\mu{|L_p^0}$ and
$Q_{\nu,p}:=P_\nu{|L_p^0}$. Then $\sigma(P_\mu,L_p)=\sigma(Q_{\mu,p})\cup\{1\}$
with $r(Q_{\mu,p})<1$, and $\sigma(Q_{\mu,2})=\overline{\{\hat\mu(n): n\ne 0\}}$.

For $f \in L_p^0$ we have $m*f=0$, so
$$
\lambda I-Q_{\mu,p}=\lambda I-aQ_{\nu,p}=a\big(\frac\lambda{a}I-Q_{\nu,p}\big ),
$$
which yields $\sigma(Q_{\mu,p})=a\sigma(Q_{\nu,p})$.

Fix $1\le p<2$, and let $\lambda \in \sigma(P_\nu,L_p)\backslash \sigma(P_\nu,L_2)$.
Since $\hat\nu(0)=1$, $\lambda \ne 1$, and
$ \lambda \notin\overline{\{\hat\nu(n): n\ne 0\}}=\sigma(Q_{\nu,2})$. Since
$1 \ne \lambda \in \sigma(P_\nu,L_p)=\sigma(Q_{\nu,p})\cup\{1\}$, we heve
$\lambda \in \sigma(Q_{\nu,p})$. Hence $a\lambda \in \sigma(Q_{\mu,p})\subset \sigma(P_\mu,L_p)$. But
$$
a \lambda \notin a\overline{\{\hat\nu(n): n\ne 0\}}=
\overline{\{\hat\mu(n): n\ne 0\}}=\sigma(Q_{\mu,2}).
$$
Since  $a\lambda \ne 1$, $a\lambda \notin \sigma(P_\mu,L_2)$.
\end{proof}

{\bf Remarks.} 1. Theorem \ref{sato2} shows that Theorem \ref{ac-power} cannot be improved.

2. By Theorem \ref{rajchman}, $\mu$ of Theorem \ref{sato2} is not Rajchman.
On the other hand, in Proposition \ref{zafran}  $\mu$ is Rajchman, so for $1<p< \infty$
the spectra of $P_\mu$ in $L_p$ are the same.

\begin{prop} \label{no-rho-L1}
There exists a continuous probability measure $\mu$ on $\mathbb T$, with all its
convolution powers singular, such that ${\lim_n \|P_\mu^n-E\|_p > 0}$ for every
$1\le p\le \infty$, while  $\sigma(P_\mu,L_1)=\sigma(P_\mu,L_2)$.
\end{prop}
\begin{proof} By \cite[Remark 5.2.5]{Ru}, every Cantor set $C \subset \mathbb T$
supports a continuous probability measure. Combining Theorems 5.2.2(a) and 5.5.2(b)
of \cite{Ru}, we obtain that there exists a Cantor set $C \subset \mathbb T$, which
is also a Kronecker set, such that every continuous probability $\mu$ supported on
$C$ satisfies $\overline{\{\hat\mu(n): n\in\mathbb Z\}}=\bar{\mathbb D}$ (where
$\mathbb D$ is the open unit disc).
Since $\|P_\mu\|_1=1$,$\ \sigma(P_\mu,L_1)=\sigma(P_\mu,L_2)$ follows from
$$
\sigma(P_\mu,L_2) =\overline{\{\hat\mu(n): n\in \mathbb Z\} } \subset
\sigma(P_\mu,L_1) \subset \bar{\mathbb D}=
\overline{\{\hat\mu(n): n\in \mathbb Z\} }=\sigma(P_\mu,L_2).
$$
Similarly, $\sigma(P_\mu,L_p)=\sigma(P_\mu,L_2)$ for $1<p< \infty$.

Since $\sigma(P_\mu,L_2)=\bar{\mathbb D}$, $\ \sup_{n\ne 0}|\hat\mu(n)|=1$,
so $\lim_n\|P_\mu^n-E\|_2 >0$ by \eqref{rho-fourier}. 
Then ${\lim_n\|P_\mu^n-E\|_p >0}$ for $p\ge 1$ follows from \cite{Ro}.

By Theorem \ref{harris}, $\mu$ and its convolution powers must be singular.
\end{proof}

{\bf Remarks.} 1. Since  $\sigma(P_\mu,L_2)=\bar{\mathbb D}$, the probability $\mu$ of Proposition \ref{no-rho-L1} is not Rajchman, and since $\sigma(P_\mu,L_2)$ is not real, 
$\mu$ is not symmetric.

2. In Proposition \ref{no-rho} we can have a probability with support $\mathbb T$.
Proposition \ref{no-rho-L1} adds the spectra equality, but the probability has a "small"
support.  In both propositions $\mu$ is not Rajchman.

3. By Lemma \ref{strict}, a non-discrete $\mu$, in particular a continuous one, 
is strictly aperiodic; hence $\|P_\mu^n f -Ef \|_p \to 0$ for every $f \in L_p$, 
$1 \le p< \infty$. (e.g. \cite{ConL}).
\bigskip

\section{Almost everywhere convergence of convolution powers}

In this section we study the almost everywhere (a.e.) convergence of of $P_\mu^nf$ for
every $f \in L_p(\mathbb T,m)$, $1\le p<\infty$, when $\mu$ is strictly aperiodic.
By Hopf's pointwise ergodic theorem (see \cite[Theorem VIII.6.6]{DS}), when $\mu$ is adapted 
we have $\frac1n\sum_{k=1}^n P_\mu^k f \to Ef$ a.e. for every $f \in L_1(\mathbb T,m)$.

We start with a general result, giving conditions for strengthening the pointwise ergodic theorem.
\begin{prop} \label{a-e}
Let $P$ be a Markov operator on $(S,\Sigma)$ with invariant probability $m$, assumed
ergodic (as in the Introduction). Fix $p \ge 1$. If $\|P^n-E\|_p \to 0$, then every
$f \in L_p(S,m)$ satisfies $P^nf \to Ef$ a.e.

If $Ef=0$, then $\sum_{k=1}^\infty P^k f/k$ converges a.e and in $L_p$-norm.
\end{prop}
\begin{proof} Put $L_p^0(m):=\{f\in L_p(m): \int f\,dm=0\}$. Then $L_p^0(m)$ is 
$P$-invariant, and $P_0:=P_{|L_p^0}$ satisfies, by assumption, $\|P_0^n\|_p \to 0$.
Then $\sigma(P_0,L_p^0) \cap \mathbb T=\{0\}$, so $r(P_0)<1$, and then 
$\|P_0^n\|_p\le Mr^n$ for some $M>0$ and $r<1$. For $f \in L_p^0(m)$ we have
$$
\int \sum_{n=0}^\infty |P^n f|\,dm =  \sum_{n=0}^\infty \int\, |P^n f|\,dm \le
 \sum_{n=0}^\infty \|P^n f\|_p\,dm = \sum_{n=0}^\infty \|P_0^n f\|_p\,dm< \infty.
$$
Hence $\sum_{n=0}^\infty |P^n f| < \infty$ a.e., so $P^nf \to 0$ a.e.

For $f \in L_p(m)$ we then have $P^nf -Ef=P^n(f-Ef) \to 0$ a.e.
\smallskip

The convergence $\|P^n-E\|_p \to 0$ implies that $(I-P)L_p(m)$ is closed, and equals
$L_p^0(m)$. Hence $Ef=0$ implies $f =(I-P)g$ with $g \in L_p(m)$. Then
\begin{equation} \label{eht}
\sum_{k=1}^n\frac{P^kf}k= \sum_{k=1}^n\frac{P^kg}k - \sum_{k=1}^n\frac{P^{k+1}g}k =
Pg -\sum_{k=2}^n\frac{P^kg}{k(k-1)}-\frac{P^{n+1}g}n.
\end{equation}
The series on the right converges in norm, and also a.e. by Beppo Levi's theorem.
The last term converges a.e. to 0 by Hopf's pointwise ergodic theorem.
\end{proof}

{\bf Remarks.} 1. The above proof of a.e. convergence of the series $\sum_{k=1}^n\frac{P^kf}k$
is valid for any $P$ with invariant probability $m$ and $f \in (I-P)L_1$. Cuny \cite{Cu}
proved that if the series converges in $L_1(m)$-norm, then it converges a.e.

2. By \cite[Theorem 3.1]{DL}, the series $\sum_{k=1}^n\frac{P^kf}k$ converges a.e. for
$P$ as in the previous remark and $f \in (I-P)^\alpha L_1$, $0<\alpha<1$. When
$(I-P)L_1$ is not closed (i.e. $\frac1n\sum_{k=1}^n P^k$ does not converge in $L_1$
operator norm), $(I-P)L_1 \subsetneq (I-P)^\alpha L_1$, by \cite{DL}.

3. The assumption $\|P^n-E\|_p \to 0$ yields a.e. convergence of $P^{n+1}g/n \to 0$ without 
using Hopf's theorem, since we can assume $Eg=0$, and then $\|P^n g\|_p \le Mr^n\|g\|_p$
yields $L_p$-norm and a.e. convergence of $\sum_{n=1}^\infty P^{n+1}g/n$.
\medskip

Applying Proposition \ref{a-e} to Theorems \ref{harris} and \ref{rajchman} we obtain

\begin{cor} \label{har2}
 Let $\mu$ be a probability on $\mathbb T$ such that $\mu^k$ is non-singular for some 
$k\ge 1$. Then for every $f \in L_1(\mathbb T,m)$ we have $P_\mu^nf \to Ef$ a.e.
\end{cor}

{\bf Remark.} The results of \cite{ConL} do not treat convergence for all $L_1$ functions.

\begin{cor}\label{r2}
 Let $\mu$ be a Rajchman probability on $\mathbb T$. Then for every $p>1$ and
$f \in L_p(\mathbb T,m)$ we have $P_\mu^nf \to Ef$ a.e.
\end{cor}

{\bf Remarks.} 1. There are Rajchman probability measures with all powers singular
\cite{CoL}, \cite{Pa}, to which Corollary \ref{r2} applies.

2. Since $\sup_{n\ne 0}|\hat\mu(n)|<1$ when $\mu$ is Rajchman, Corollary \ref{r2} follows
also from \cite[Theorem 2.4]{ConL}.

3. The classical Cantor-Lebesgue measure $\mu$ is continuous with singular powers, and is not
 Rajchman, but  $\sup_{n\ne 0}|\hat\mu(n)|<1$ \cite{DL4}; hence $P_\mu^nf \to Ef$ a.e for 
any $f \in L_p(\mathbb T,m)$, $p>1$.

4. When $\mu$ is a symmetric strictly aperiodic probability, $P_\mu$ is symmetric on $L_2$, 
and $P_\mu^nf \to Ef$ a.e. for  every $f \in L_p(\mathbb T,m)$, $p>1$, by Stein's theorem
 \cite{St}.

5. If $\mu$ is strictly aperiodic, then $|\hat\mu(n)|<1$ for $n \ne 0$, which yields that 
$P_\mu^n f(x)\to Ef$ uniformly for every trigonometric polynomial $f$; hence for every
continuous $f$ we have $\|P_\mu^nf -Ef\|_{C(\mathbb T)}\to 0$. It follows that 
$P_\mu^*\nu=\nu \in C(T)^*=M(\mathbb T)$ implies $\nu=cm$ for some $c \in\mathbb C$. 
By the ergodic decomposition for mean ergodic operators, 
$\overline{(I-P_\mu)C(\mathbb T)}= \{f \in C(\mathbb T): Ef=0\}$.

6. Let $\mu$ be strictly aperiodic. Then for $0\ne k \in\mathbb Z$, $e_k(z):=z^k$ is in 
$(I-P)C(\mathbb T)$. If $f$ is a trigonometric polynomial with integral zero, i.e. 
$f =\sum_{0<|k| \le N} a_ke_k$, then $f \in (I-P)C(\mathbb T)$, and
$\sum_{n=1}^\infty \frac{P^n_\mu f(x)}n$ converges uniformly on $\mathbb T$, 
by \eqref{eht}.

\begin{prop} \label{EHT}
Let $\mu$ be a strictly aperiodic probability on $\mathbb T$. Then the following are 
equivalent.

(i) The series {(\rm one-sided ergodic Hilbert transform)}
\begin{equation}\label{eht-series}
\sum_{k=1}^\infty \frac{P_\mu^k f(x)}k
\end{equation}
 converges uniformly for every $f \in C(\mathbb T)$ with $Ef=0$.
\smallskip
 
(ii) \quad ${ \|\frac1n\sum_{k=1}^n P_\mu^k -E\|_{L_\infty} \to 0}$.
\smallskip

(iii) \quad ${\| P_\mu^n -E\|_{L_\infty} \to 0}$.

(iv) $\ \mu$ has a non-singular power.

(v) \quad ${\| P_\mu^n -E\|_{C(\mathbb T)} \to 0}$.

(vi) The series $\sum_{k=1}^\infty \frac1k P_\mu^k f$ converges weakly in $C(\mathbb T)$ 
	for every $f \in C(\mathbb T)$ with ${Ef=0}$.

(vii)   ${\ \| P_\mu^n -E\|_{B(\mathbb T)} \to 0}$  \text{\rm ($B(\mathbb T)$ is the space 
of all bounded Borel functions)}.
\end{prop}
\begin{proof}
(ii) implies (i). Since $P_\mu$ is a contraction of $C(\mathbb T)$, and 
$\sigma(P_{\hat\mu},L_1)=\sigma(P_{\hat\mu},M(\mathbb T))$ by \cite{La}, (ii) is 
equivalent to $\|\frac1n\sum_{k=1}^nP_\mu^{*k}-E^*\|_{M(\mathbb T)} \to 0$, 
which is equivalent to $ \|\frac1n\sum_{k=1}^n P_\mu^k -E\|_{C(\mathbb T)} \to 0$. 
Then $(I-P_\mu)C(\mathbb T)$ is closed, and every $f \in C(\mathbb T)$ with $Ef=0$
is in $(I-P_\mu)C(\mathbb T)$, so the uniform convergence of the series 
\eqref{eht-series} follows from \eqref{eht}.
\smallskip

(i) implies (ii). We assume the uniform convergence of \eqref{eht-series} for every 
continuous $f$ with $Ef=0$. By an operator theoretic result 
\cite[Theorem 1.1]{CoL1}, $(I-P_\mu)C(\mathbb T)$ is closed. 
Let $Q_\mu$ be the restriction of $P_\mu$ to ${\mathbf C}:=(I-P_\mu)C(\mathbb T)$. 
Since $P_\mu$ is mean ergodic,
$$
\Big\|\frac1n\sum_{k=1}^nP_\mu^k -E\Big\|_{C(\mathbb T)} =
\Big\|\frac1n\sum_{k=1}^nP_\mu^k(I-E) \Big\|_{C(\mathbb T)} \le 
\Big\|\frac1n\sum_{k=1}^nQ_\mu^k\Big\|_{\mathbf C}\|I-E\|_{C(\mathbb T)} \to 0.
$$
Going through the adjoints as above we obtain 
$ \|\frac1n\sum_{k=1}^n P_\mu^k -E\|_{L_\infty} \to 0$.
\smallskip

(ii) implies (iii). Since the space of fixed points of $P_\mu$ on $L_\infty$ is 
one-dimensional, by \cite{L} $P_\mu$ is quasi-compact on $L_\infty$. Hence 
$\sigma(P_\mu,L_\infty)\cap\mathbb T$ consists only of finitely many poles (which
are eigenvalues) \cite[p. 711]{DS}. By strict aperiodicity $P_\mu$ has no unimodular 
eigenvalues except 1, so $\|P_\mu^n-E\|_{L_\infty} \to 0$.

Obviously (iii) implies (ii).
\smallskip

The equivalence of (iv) and (iii) is by Corollary \ref{doeblin}.
\smallskip

 Since $\sigma(P_{\hat\mu},L_1)=\sigma(P_{\hat\mu},M(\mathbb T))$, we have
$\|P_\mu^{*n}-E^*\|_{M(\mathbb T)} \to 0$ if and only if
${\|P_\mu^{*n}-E\|_{L_1}} \to 0$, which proves the equivalence of (iii) and (v).
\smallskip

The equivalence of (i) and (vi) follows from \cite[Theorem 1.1]{CoL1}.
\smallskip

(iv) implies that $P_\mu$ is quasi-compact on $B(\mathbb T)$, by 
\cite[Proposition VIII.3]{BR}. By strict aperiodicity $P_\mu$ on $B(\mathbb T)$ has no 
unimodular eigenvalues except 1, so by \cite[ p. 711]{DS} (vii) holds.

Clearly (vii) implies (v).
\end{proof}

{\bf Remark.} If a power-bounded operator $T$ on a Banach space $X$ is uniformly ergodic,
i.e. $\frac1n \sum_{k=1}^n T^k$ converges in operator norm, then $(I-T)X$ is (easily
shown to be) closed, and by \eqref{eht} $\sum_{k=1}^n \frac1k T^kv$ converges for
every $v$ in $(I-T)X=\overline{(I-T)X}$; this yields a direct proof of (v) implies (i) in
Proposition \ref{EHT}. It is proved in \cite{CoL2} that if $\sum_{k=1}^n \frac1k T^kv$ 
converges for every $v$ in $\overline{(I-T)X}$, then $T$ is uniformly ergodic.

\begin{prop} \label{sso}
 There exists a continuous strictly aperiodic probability $\mu$ on $\mathbb T$ (necessarily
with all its powers singular) such that for some non-trivial $A \in \mathcal B$ we have
\begin{equation} \label{SSO}
\limsup_n P_\mu^n1_A=\limsup_n\check\mu^n*1_A= 1 \ a.e.  \ \text{\rm and } \ 
\liminf_n P_\mu^n1_A=\liminf_n \check\mu^n*1_A =0\ a.e.;
\end{equation}
 hence a.e $P_\mu^n1_A$ does not converge.
\end{prop}
\begin{proof} Let $\mu$ be a continuous probability on $\mathbb T$ with
$\overline{\{\hat\mu(n): n\in\mathbb Z\}}=\bar{\mathbb D}$, as in the proof of  
Proposition \ref{no-rho-L1}. By Lemma \ref{strict}, continuity implies strict aperiodicity.

Since $\widehat{\check\mu}(n)=\hat\mu(-n)$ and $\overline{\{\hat\mu(n): n\in\mathbb Z\}}$ 
contains the unit circle, by \cite[Corollary 3.2]{ConL}, both $\mu$ and $\check\mu$ have 
the "strong sweeping out property", which for $\check\mu$ means that there is a $G_\delta$ 
dense subset of $\mathcal B$ consisting of non-trivial $A$ satisfying  \eqref{SSO}.
\end{proof}

{\bf Remarks.} 1. By \cite{Roj}, the strictly aperiodic discrete  probability 
$\mu:= \frac12(\delta_1+\delta_z)$, with $z \in \mathbb T$ not a root of unity, 
satisfies  the strong sweeping out property.

2.  Proposition \ref{sso} was proved somewhat differently in \cite[Proposition 3.4]{ConL}.
\medskip

{\bf Example.} {\it A probability $\mu$  with singular powers and $P_\mu^nf \to Ef$ a.e. 
$\forall f \in L_1$.}

\noindent
Let $z \in \mathbb T$ which is not a root of unity (i.e. $z^n \ne 1$ for $n\ge 1$). Then 
$\tau x:=zx$. $x \in\mathbb T$ is an invertible ergodic measure preserving transformation 
of $(\mathbb T,m)$. Define $\mu:=\frac13(\delta_{z^{-1}}+\delta_1 +\delta_z)$. By symmetry,
$P_\mu f=\mu*f=\frac13(f\circ\tau^{-1}+f+f\circ\tau)$. By \cite[Corollary 1.2]{BJR},
for every $f \in L_1(\mathbb T,m)$ we have $P_\mu^n f\to \int f\,dm$ a.e.
\smallskip

A generalization of the example, is the following. Let $(\mu_k)_{k \in \mathbb Z}$ satisfy 
$$ 
(i)\ \mu_k \ge 0 \text{ and }\sum_{k=-\infty}^\infty \mu_k=1. \quad 
(ii)\ \sum_{k=-\infty}^\infty k^2\mu_k < \infty. \quad 
(iii)\ \sum_{k=-\infty}^\infty k\mu_k =0.
$$
For $z$ and $\tau$ as above, put $\mu= \sum_{k=-\infty}^\infty \mu_k\delta_{z^k}$.
By \cite{BC}, $P_\mu^nf$ converges a.e. for every $f \in L_1(\mathbb T,m)$ 
(for $f \in L_p$, $p>1$, a.e. convergence was proved in \cite{BJR});
when $\mu_0 \ne 1$,  $\mu$ is adapted and the limit is $Ef$. Wedrychowicz \cite{We}
gave sufficient conditions on $(\mu_k)$ for the above a.e. convergence for every 
$f \in L_1$ when $\sum_{k=-\infty}^\infty k^2\mu_k = \infty$;  non-symmetric
examples are presented.
\smallskip

{\bf Remark.} In the example $\mu$ is discrete, so $\lim_n\|P_\mu^n-E\|_p>0$ 
for every $1 \le p \le \infty$, since $\lim_n \|P_\mu^n-E\|_2>0$ by \cite{DL4}, 

\medskip

It is known \cite{Or} that in general Stein's theorem fails for $p=1$, 
but a.e convergence holds  if $\int |f| \log ^+|f| dm< \infty$, by Rota's theorem 
(see \cite[p. 92]{Starr}).

\begin{theo} \label{os}
Le $\mu$ be symmetric and strictly aperiodic on $\mathbb T$. Then for every 
$f \in L_1(\mathbb T,m)$, $\ P_\mu^n f \to Ef$ a.e.
\end{theo}
\begin{proof} We will deduce our result from the general work of Oseledets \cite{Os}.
We look at the action of $\mathbb T$ by rotations on $(\mathbb T,\mathcal B,m)$.
The action is ergodic, and by strict aperiodicity condition (2) of \cite{Os} is satisfied.
By \cite[Theorem 3]{Os}, $P_\mu^nf$ converges a.e. whenever $f \in L_1(\mathbb T,m)$.
By  Hopf's pointwise ergodic theorem,
$\lim_n P_\mu^n f =\lim_N \frac1N \sum_{n=1}^N P_\mu^n f = \int f\,dm$  a.e.,
since $P_\mu$ is ergodic.
\end{proof} 

\begin{cor} \label{cl-ae}
Let $\mu$ be the classical Cantor-Lebesgue measure. Then for every 
$f \in L_1(\mathbb T,m)$, $\mu^n*f \to \int f\,dm$ a.e.
\end{cor}

{\bf Example} {\it A continuous symmetric $\mu$ with singular powers, such that 
$P_\mu^n f \to Ef$ a.e. $\forall f \in L_1$, but $\sigma(P_\mu,L_p) \ne \sigma(P_\mu,L_2)$
for $1\le p <2$.}

\noindent
Let $\mu$ be Sato's continuous symmetric  probability measure of Proposition \ref{sato}.
Then $P_\mu^nf \to Ef$ a.e. for any $f \in L_1$  by Theorem \ref{os}.
\smallskip

{\bf Remark.} As mentioned earlier, in the example the peripheral spectra
$\sigma(P_\mu,L_p)\cap\mathbb T$ are all the same for $1<p<\infty$, by
\cite[Theorem 4.3]{CL26}, but this is not true for $p=1$.

\medskip

{\bf Example} {\it A continuous symmetric $\mu$ with singular powers such that
 $P_\mu^nf \to Ef$ a.e. $\forall f \in L_1$, but $\lim_n\|P_\mu^n-E\|_p > 0$ for $p\ge 1$.}

\noindent
Let $\mu_0$ be the continuous probability measure of the proof of Proposition \ref{no-rho-L1},
and put $\mu:=\frac12(\mu_0+\check\mu_0)$. Then $\mu$ is symmetric and $P_\mu$ is
symmetric on $L_2(m)$. Since $\{\hat\mu_0(n): n\in\mathbb Z\}$ is dense in $\bar{\mathbb D}$,
$\{\hat\mu(n): n\in\mathbb Z\}$ is dense in $[-1,1]$, so $\sigma(P_\mu,L_2)=[-1,1]$;
hence $1$ is not isolated in the spectrum, and $\lim_n\|P_\mu^n-E\|_2>0$ (so necessarily
all powers of $\mu$ are singular). Hence $\lim_n\|P_\mu^n-E\|_1>0$ by Proposition
 \ref{singular}, and $\lim_n\|P_\mu^n-E\|_p>0$ for $1<p< \infty$ by \cite{Ro}. 
By Theorem \ref{os}, $P_\mu^nf \to Ef$ a.e for $f \in L_1(\mathbb T,m)$.
 \medskip

{\bf Problem.} {\it Are there necessary and sufficient conditions for $\mu$ strictly aperiodic,
 in terms of $\{\hat\mu(n)\}$,
for $P_\mu^n f \to Ef$ a.e. for every $f \in  L_1(\mathbb T,m)$?}

\noindent
In view of Corollary \ref{har2}, the Problem is when all powers of $\mu$ are singular.
Necessary and sufficient conditions for $P_\mu^n f \to Ef$ a.e. for every 
$f \in  L_p(\mathbb T,m)$,  $\ p>1$, were given in \cite[Theorem 3.6]{ConL}; 
one of them is that for some $j\ge 1$ we have 
\begin{equation} \label{conze}
\sup_{0\ne n \in\mathbb Z} \frac{|1-\hat\mu(n)^j|}{1-|\hat\mu(n)|^j} < \infty.
\end{equation}
Hence \eqref{conze}, with some $j\ge 1$, is necessary for the a.e convergence
$P_\mu^n f \to Ef$  for every $f \in L_1$; is it sufficient? 
Note that \eqref{conze} is satisfied, with $j=1$, when $\|P_\mu^n-E\|_2 \to 0$.
\smallskip

{\bf Remarks.} 1. It was noted in \cite[p. 418]{BJR} that \eqref{conze} with $j=1$
is equivalent to $\{ \hat\mu(k): k\in\mathbb Z\}$, and therefore $\sigma(P_\mu,L_2)$,
being contained in some closed Stolz region\footnote{A closed Stolz region is the closed 
convex hull of $1$ and a disk of radius $r<1$ centered at 0.}.

2. If \eqref{conze} holds for $j=1$, then it holds also for every $j>1$.
However, the example on top of \cite[p. 558]{ConL} shows that \eqref{conze} may hold for
$j=2$ and not for $j=1$. In that example $\mu$ is symmetric, so $P_\mu^nf \to Ef$ a.e. for
every $f \in L_1$, by Theorem \ref{os}.
\smallskip

A special case of  of the problem is: {\it does \eqref{conze} with $j=1$ imply
$P_\mu^nf \to Ef$ a.e. for every $f \in L_1(m)$?}
In particular, {\it if $\mu$ is any Rajchman probability (with singular powers), 
do we have a.e. convergence of $P_\mu^nf$ for every integrable function $f$? }
\bigskip

{\bf Acknowledgements.} The authors thank J.-P. Conze for bringing Borda's paper 
\cite{Bo} to their attention, and are grateful to Y. Derriennic for his helpful comments
on the connections with the paper of Brunel-Revuz \cite{BR}.

\end{document}